\documentclass{amsart}

\usepackage{amsmath,amsfonts,amssymb}
\newtheorem{teo}{Theorem}
\newtheorem{rk}{Remark}

\newtheorem{lemma}{Lemma}

\newcommand{\R}{{\mathbb{R}}}
\newcommand{\C}{{\mathbb{C}}}
\newcommand{\Z}{{\mathbb{Z}}}
\newcommand{\N}{{\mathbb{N}}}

\newcommand{\ta}{TA}

\newcommand{\homeo}{Homeo}
\begin{document}
\title{Linearization of topologically Anosov homeomorphisms of non compact  surfaces.}
\author{Gonzalo Cousillas, Jorge Groisman and Juliana Xavier}

 \begin{abstract}  We study the dynamics of {\it Topologically Anosov} homeomorphisms of non compact surfaces.  In the case of surfaces of genus zero and finite type, 
 we classify them. We prove that if $f:S \to S$, is a Topologically 
 Anosov homeomorphism
  where
 $S$ is a non-compact surface of genus zero and finite type, then $S= \R ^ 2$ and $f$
 is conjugate to a homothety or reverse homothety (depending on wether $f$ preserves or reverses orientation). A weaker version of this result was conjectured in \cite{cgx}.

 \end{abstract}
 \maketitle

 \section{Introduction}

Let $f: S \to S$ be a homeomorphism and $\delta: S\to \R $ a continuous and strictly
  positive function.  A $\delta$-{\it pseudo-orbit} for $f$ is a sequence $(x_n)_{n\in \Z}\subset S$ such that $d(f(x_n), x_{n+1})< \delta (f(x_n))$ for all $n\in \Z$.
  If $\epsilon: S\to \R $ a continuous and strictly
  positive function, then a $\delta$-pseudo-orbit $(x_n)_{n\in \N}$ is {\it $\epsilon$-shadowed} by an orbit, if there exists $x\in S$ such that $d(x_n ,f^n (x))< \epsilon (x_n)$ for
  all $n\in \Z$.

 Throughout this paper $f:S\to S$ is a {\it Topologically Anosov} ($\ta$) homeomorphism.  That is:
 \begin{itemize}
  \item it is {topologically expansive}: there exists a continuous and strictly positive function $\epsilon: S\to \R$ such that for all $x, y \in S, x\neq y$ there exists $k\in \Z$ 
  satisfying $d(f^k (x),f^k (y)) > \epsilon (f^k (x))$;
  \item it satisfies the {\it topological shadowing property}: for all continuous and strictly positive function $\epsilon: S\to \R$ there exists $\delta: S\to \R $ a continuous and strictly
  positive function such that every $\delta$-pseudo-orbit is $\epsilon$-shadowed by an orbit. 
 \end{itemize}
 
 These definitions are generalizations of the classic notions of {\it uniform} expansivity and pseudo-orbit tracing property, suited for non-compact metric spaces. 
 On non-compact spaces it is well known that a dynamical system may be expansive or have the shadowing property with respect to one metric, but not with respect to another metric that 
 induces the same topology. Topological definitions of expansiveness and shadowing were given in \cite{dlrw} for first countable, locally compact, paracompact and Hausdorff topological
 spaces; they are equivalent to the usual metric definitions for 
 homeomorphisms on
 compact
 metric spaces, but are independent of any change of compatible metric. The definitions we gave correspond to those given in  \cite{dlrw} in the case of metric spaces, and appeared
 fist in literature in \cite{lny}.

 To illustrate what happens in the non-compact setting, note that a rigid translation in the plane is topologically expansive but does not satisfy the topological shadowing property. An example of $\ta$ homeomorphism is any homothety
 (or reverse homothety) in 
 $\R ^ 2$
 (see \cite{cou} for a proof).  As being $\ta$ is a conjugacy invariant, the whole conjugacy class of homotheties belongs to the family of $\ta$ homeomorphisms.  In this work we deal
 with the problem of classifying $\ta$ homeomorphisms.  In particular, are all $\ta$ plane  homeomorphisms conjugate to a homothety (or reverse homothety)?  Are these the only 
 examples for $\ta$ 
 homeomorphisms of non compact surfaces?\\
 
 We prove the following:
 
 \begin{teo}\label{t1} Let $f:S \to S$, be a Topologically Anosov homeomorphism,
  where
 $S$ is a non-compact surface of genus zero and finite type. Then $S= \R ^ 2$ and $f$
 is conjugate to a homothety or reverse homothety (depending on wether $f$ preserves or reverses orientation). 
  
 \end{teo}
 
By {\it reverse homothety} we mean the map $z\mapsto \overline z /2$, $z\in \C$.\\

Particular cases of this problem were treated on \cite{cgx}, and a weaker version of the theorem stated above is contained in Conjecture 1.1 of that paper. We also refer the reader to 
the former reference 
for a characterization of Topologically Anosov homeomorphisms on $\R$.

Expansive homeomorphisms with 
 the shadowing property on compact metric spaces are known to have spectral decomposition in Smale's sense (\cite{aoki}).

We use the spectral decomposition for Topologically Anosov homeomorphisms in \cite{dlrw}:

\begin{teo}\label{esp}  Let $X$ be a first countable, locally compact, paracompact, Hausdorff space and $f:X\to X$ a $\ta$ homeomorphism.  Then, $\Omega (f)$ can be written as a union of
disjoint closed invariant sets on which $f$ is topologically transitive.
 
\end{teo}

The disjoint closed invariant sets given by the previous theorem are called {\it basic sets}.

\section{Sufficient condition}

We show in this section that in the presence of an attracting (or repelling) fixed point, $S= \R ^ 2$ and $f$
 is conjugate to a homothety or reverse homothety.  This result is general and makes no assumption on the topology of $S$, which can be {\it any} non-compact surface.  By attracting
 fixed point we mean a fixed point $x_0$ with an open neighborhood $U$ such that $\overline{f(U)}\subset U$ and $\cap_{n\geq 0} f^n (U) = \{x_0\}$.

If $\epsilon: S\to \R$ is a strictly positive continuous function, we denote $W^s_\epsilon (x)=\{y\in S: d(f^n (x), f^n (y))<\epsilon (f^n (x)),  \forall n \geq 0\}$, and
$W^u_\epsilon (x)=\{y\in S: d(f^n (x), f^n (y))<\epsilon (f^n (x)),  \forall n \leq 0\}$.  By Proposition 20 in \cite{dlrw}, if $\epsilon: S\to \R$ is given by topological
expansivity, then there exists $\delta:S\to \R$ a continuous strictly positive function and a continuous map $t: B(x, \delta (x))^ 2\to S$ such that 
$$W^s_\epsilon (w)\cap W^u_\epsilon (y)=\{t(w,y)\}\, $$\noindent if $(w,y) \in B(x, \delta (x))^ 2$.  As a consequence, the set  $t(\gamma(s)\times \{y\}), s\in [0,1]$ is a connected
set contained in $W^u_\epsilon (y)$,
where 
$\gamma:[0,1]\to B(x,\delta(x))$ is a continuous map such
that $\gamma (0)=x$ and $\gamma (1) = y$.

\begin{lemma}\label{suf}  If there exists an attracting fixed point $x_0$, then $S=\R ^ 2$  and $f$
 is conjugate to a homothety or reverse homothety (depending on wether $f$ preserves or reverses orientation). 
 
\end{lemma}

\begin{proof} Let $B$ be the basin of attraction of $x_0$, and note that it is enough to show that $B=S$. Indeed, if $B=S$, then $S$ is a nested increasing union of disks,
therefore a simply connected non compact 
surface, so it is homeomorphic to $\R^2$.  Moreover, if $B=S$, then $x_0$ is globally asymptotically stable, and therefore $f$
 is conjugate to a homothety or reverse homothety by the cassic Ker\'ekj\'art\'o's theorem (\cite{k1},\cite{k2}).
Take $x\in \partial B$. By the remarks preceeding this lemma,  there exists a 
neighbourhood $U$ of $x$ such that 
$W^s_{\epsilon} (w) \cap W^u_{\epsilon} (y)=t(w,y) $ for all $(w,y)\in U^2$, where  $\epsilon: S\to \R$ is given by topological
expansivity. Note that this is still true if we change the function $\epsilon$ for some other $\epsilon': S\to \R$ continuous and strictly positive such that
$\epsilon' (z)\leq \epsilon(z)$ for all $z\in S$.  Modifying the function $\epsilon$ if necessary, we have that any $y\in U\cap B$ does not belong to $W^s_\epsilon (x)$, because
$y\in B$ implies $f^n (y)\to_{n\to \infty} x_0$. Therefore we may take $\gamma:[0,1]\to U$ a continuous map such
that $\gamma (0)=x$ and $\gamma (1) = y$ and the set $t(\gamma(s)\times \{y\}), s\in [0,1]$ is not reduced to $\{y\}$.
We conclude that
$W^u_{\epsilon} (y)$ contains a connected set properly containing $y$.  Then, we can find  
$z\in W^u_{\epsilon} (y)\cap B$, $z\neq y$, such that $z\in W^s_{\epsilon} (y)$ (because if $z$ is sufficiently close to $y$, its iterates will remain close to those of $y$ until
they are both on $B(x_0, \epsilon (x_0)/2)$).  Then, $z\in W^s_\epsilon (y)\cap W^u_\epsilon (y) $ contradicting expansivity .

\end{proof}

 \section{Unbounded orbits}
  The following lemma is a  generalization of Lemma 2 in 
\cite{cgx}.  As in the previous section, we make no assumptions on the topology of $S$, which is considered to be {\it any}
non-compact surface.

\begin{lemma}\label{epsilon}
Let $f\in \homeo (S)$. If there exists $x\in S$ and a sequence of positive integers $(n_k)_{k\in \N}$ such that $f^{n_k}(x)\to_{k\to \infty} \infty$  then there exists $\epsilon: S \to \R$ a 
continuous positive map with the property that if $y\neq x$, then there 
exists 
$ k\in \N$ such that 
$d(f^{n_k} (x),f^{n_k} (y)) \geq \epsilon (f^{n_k} (x))$.  
 
\end{lemma}

\begin{proof}

First note that there exists a family of pairwise disjoint open sets $(U_{k})_{k\in \N}$ such that each $U_{k}$ is a neighbourhood of $f^{n_k} (x)$. We let $U_{0}$ be a
neighbourhood of $x$ and set $n_0 = 0$.
We claim that there exists a family of open sets $(V_{k})_{k\in \N}$ such that for all $k\in \N$, $V_{k} \subset U_{k}$, $f^{n_k}(x)\in V_{k}$, and a continuous map 
$h:\cup_k V_{k}\to \R^2$ which is a homeomorphism onto its image such that $hf^{n_k-n_{k-1}}|_{V_{k}}=T^{n_k-n_{k-1}}h$ for all $k\in \N$, where $T(x,y) = (x+1, y)$ for all $(x,y)\in \R ^2$.
Take a homeomorphism $h: U_0 \to B((0,0), 1/3)$, and let $V_0\subset U_0 $ an open set containing $x$ such that $f^{n_1}(V_0)\subset U_1$.  Define $\tilde U_1 := f^{n_1}(V_0)$ 
and extend the homeomorphism
$h$ to $\tilde U_1$ as $h|_{\tilde U_1} = T^{n_1}hf^{-n_1}$.  Note that $hf^{n_1}|_ {V_0} = T^{n_1}h|_{V_0}$. We now define $V_1\subset \tilde U_1$ such that $f^{n_2-n_1}(V_1)\subset U_2$,
let $\tilde U_2 = f^{n_2-n_1} (V_1)$
and extend $h$ to $\tilde U_2$ as $h|_{\tilde U_2} = T^{n_2-n_1}hf^{-n_2+n_1}$.  Inductively, if $h$ is defined on $\tilde U_i\subset U_i$, we extend $h$ to 
$\tilde U_{i+1}\subset U_{i+1}$ as
follows.
We take $V_i \subset \tilde U_i$ such that $f^{n_{i+1}-n_i}(V_i) \subset U_{i+1}$ and let $\tilde U_{i+1} = f^{n_{i+1}-n_i}(V_i)$.  We then let $h|_{\tilde U_{i+1}} = 
T^{n_{i+1-n_i}}h f^{-n_{i+1}+n_i}$.  Note that for all $i$, 
$hf^{n_{i+1}-n_i}|_ {V_i} = T^{n_{i+1}-n_i}h|_{V_i}$.  This proves the claim. 

Now take $\tilde \epsilon: \R^ 2 \to \R$ a continuous positive map verifying that for all $n\in \N$, $B((k,0), \tilde \epsilon ((k,0)))\subset h (V_k)$ and also that if $y\neq x$, 
then there exists 
$ n_0>0$ such that 
$||T^n (x)- T^n (y)|| > \tilde \epsilon (T^n (x))$ for all $n\geq n_0$.  Finally, we define $\epsilon: S \to \R$ such that $B(f^{n_k} (x),\epsilon (f^{n_k} (x)))\subset h^{-1}
(B((n_k,0)),\tilde \epsilon ((n_k,0)))$ 
and extend it
to a continuous positive map of $S$.  Now notice that if for some $y\neq x$, 
$d(f^{n_k} (x),f^{n_k} (y)) < \epsilon (f^{n_k} (x))$ for all $k\in \Z$, then $f^{n_k}(y) \in V_k$ for all $k\in \N$ which implies that  $T^{n_k}h(y)=hf^{n_k}(y)$ for all $k\in \N$.
So, 
$d(T^{n_k}(h(y)), (n_k,0))< \tilde \epsilon (n_k,0)$ for all $k\in \N$, a contradiction.

\end{proof}

Of course, we have an analogous statement if there exists a sequence of negative
integers $(n_k)_{k\in \N}$ such that $f^{n_k}(x)\to_k \infty$. In this case, there exists $\epsilon: S \to \R$ a 
continuous positive map with the property that if $y\neq x$, then there 
exists an integer 
$ n<0$ such that 
$d(f^n (x),f^n (y)) > \epsilon (f^n (x))$.  

\begin{rk}\label{rk1} Let $f^{n_k}(x)\to_k \infty$  for some sequence of positive integers $(n_k)_{k\in \N}$, and let $\epsilon: S \to \R$ be given by the previous lemma. Note that if
$(x_n)_{n\in \Z}$ is a pseudo-orbit such that 
$x_{n_k} = f^{n_k} (x)$ for all $k\geq k_0$, then the only possible orbit that
$\epsilon$-shadows $(x_n)_{n\in \Z}$ is that of $x$.
 \end{rk}

\begin{lemma}\label{aw}  Let $f\in \homeo (S)$ be $\ta$. For any point $x\in S$ either its forwards or backwards orbit is bounded.
 
\end{lemma}

\begin{proof}  Suppose there exists $x\in S$  such that both its forwards and backwards orbit are unbounded.  By
Lemma \ref{epsilon}, there exists $\epsilon: S \to \R$ a continuous positive map with the property that if $y\neq x$, then there exists 
$ n\in \Z,\  n>0$ such that 
$d(f^n (x),f^n (y)) > \epsilon (f^n (x))$ and $ m\in \Z, m<0$ such that 
$d(f^m (x),f^m (y)) > \epsilon (f^m (x))$.  Take $\delta: S \to \R$ a continuous positive map as in the 
definition of shadowing, and consider the following $\delta$-pseudo-orbit $(x_n)_{n\in \Z}$:  $x_n = f^{n} (x)$ for all $n<0$; $x_n= f^n (y)$ for all $n\geq 0$, where 
$y\in B (x, \delta (x)), y\neq x$.  Then, as explained in the previous remark, the orbit of $x$ is the only possibility to $\epsilon$-shadow this pseudo-orbit. However,
this is impossible  because as stated at the beginning of this proof, there exists 
$ n\in \Z,\  n>0$ such that 
$d(f^n (x),f^n (y)) > \epsilon (f^n (x))$.

\end{proof}

The previous lemma has immediate consequences on the structure of basic sets (see Theorem \ref{esp}):

\begin{lemma} The basic sets are compact.
 
\end{lemma}

\begin{proof}  Suppose there is a non-compact basic set $\Lambda$. By transitivity, there exists $x\in \Lambda$ such that both its forwards and backwards orbit are dense in $\Lambda$ 
and hence unbounded.  We are done by the previous lemma.
 
\end{proof}

We say that there is a {\it cycle} of basic sets if there exist $\Lambda _0, \ldots, \Lambda_{n-1}$ pairwise disjoint basic sets and points $x_0, \ldots, x_{n-1}$ such that 
$x_i\notin \Omega(f)$ for all $i=0, \ldots, n-1$ and such that $\alpha(x_i)\subset \Lambda_i$, $\omega(x_i)\subset \Lambda_{i+1}$ for all $i=0, \ldots, n-1$, where 
$\Lambda_n=\Lambda_0$ .

\begin{lemma}\label{cyc} There are no cycles of basic sets.
 
\end{lemma}

\begin{proof}  Take $\Lambda_i$ and $x_i$ as in the definition of cycle, $i=0, \ldots, n-1$.  Take $r>0$ such that $B(x_i, r)$, $i=0, \ldots, n-1$.  are pairwise disjoint, and 
$B(x_i, r)\cap \Omega(f) = \emptyset$, $i=0, \ldots, n-1$.  Take
$\epsilon: S \to \R$ a continuous positive map such that $\epsilon(x_i) = r$, $i=0, \ldots, n-1$.  Take $\delta: S \to \R$ a continuous positive map as in the 
definition of shadowing. For all $i=0, \ldots, n-1$, let $x_i^-\in \Lambda_i\cap \alpha(x_i)$, 
$x_i^+\in \Lambda_{i+1}\cap \omega(x_i)$; $B_i^-= B(x_i^-, \delta(x_i^-)/2)$, $B_i^+= B(x_i^+, \delta(x_i^-)/2)$; $l_i<0$ such that $f^{l_i}(x_i)\in B_i^-$, $m_i>0$ such that $f^{m_i}(x_i)\in B_i^+$, and
$n_i>0$ such that $f^{n_i}(z_i)\in B_{i+1}^-$, $z_i \in B_i^+$.

 Now construct a periodic $\delta$-pseudo-orbit as follows: $f^j (x_0), j=l_0, m_{0}-1$, $f^j (z_0), j=0, \ldots, n_0-1$, $f^j (x_{1}), j=l_{1}, m_{1}-1$, $f^j (z_{1}), 
 j=0, \ldots, n_{1}-1$, $\ldots,$ $f^j (x_i), j=l_i, m_{i}-1$, $f^j (z_i), j=0, \ldots, n_i-1$, $f^j (x_{i+1}), j=l_{i+1}, m_{i+1}-1$, $f^j (z_{i+1}), 
 j=0, \ldots, n_{i+1}-1$, $\ldots,$ $f^j (x_{n-1}), j=l_, m_{n-1}-1$, $f^j (z_{n-1}), j=0, \ldots, n_{n-1}-1$, $f^{l_0} (x_{0}).$
 
 Take $x$ such that the orbit of $x$ $\epsilon$-shadows that pseudo orbit. Then, the orbit of $x$ must intersect $B(x_0,r)$ infinitely many times, and therefore have an accumulation point
 on $B(x_0,r)$, which is a contradiction because we had chosen $B(x_0, r)\cap \Omega(f) = \emptyset$.
\end{proof}

\begin{lemma} If there exists an unbounded past-orbit, then its $\omega$-limit is contained in an attracting basic piece.
 
\end{lemma}

\begin{proof} Let $x\in S$ and a sequence of negative
integers $(n_k)_{k\in \N}$ such that $f^{n_k}(x)\to_k \infty$. By Lemma \ref{epsilon}, there exists $\epsilon: S \to \R$ a 
continuous positive map with the property that if $y\neq x$, then there 
exists an integer 
$ n<0$ such that 
$d(f^n (x),f^n (y)) > \epsilon (f^n (x))$.  Let $\delta: S\to \R$ as in the definition of shadowing.By Lemma, \ref{aw}, $\omega(x)$ is compact.  It is obvious that $\omega(x)$ cannot 
be contained in a repelling basic piece, so
we have to discard $\omega(x)\subset \Lambda$, with $\Lambda$ a saddle type piece. If that would be the case, there would exist $y\in S, y \notin \Omega(f)$ such that $\alpha(y)\subset \Lambda$.  Let
$r>0$ be such that $B(y,r)\cap {\mathcal O(x)}=\emptyset$.  Modifying the function $\epsilon$ if necessary, we may assume that $\epsilon(y) <r$. Let $z\in \Lambda$ and $n$  large 
enough such that $d(f^{-n}(y), z)<\delta (z)$. Let $w\in \Lambda$ be such that its forwards orbit is dense in $\Lambda$ and $m$ large enough such that $d(f^m (x), w)< \delta(w)$.

Now, 
construct a $\delta$-pseudo orbit $(x_n)_{n\in \Z}$ as follows:  $x_n = f^n (x)$ for all $n\leq m-1$, $x_{m+j}=f^j (w)$ for all $j=0, \ldots, l-1$, where $d(f^l (w),z) <\delta(z)/2$,
$x_{m+l+j} = f^{-n+j} (y)$ for all $j\geq 0$.  By construction of the function $\epsilon$, this pseudo-orbit must be shadowed by the orbit of $x$, but this contradicts the fact that
${\mathcal O(x)}\cap B(y, \epsilon (y)) = \emptyset$.
 
\end{proof}

\begin{lemma} There exists either an attracting basic piece or a repelling basic piece.
 
\end{lemma}

\begin{proof}  Suppose that every basic piece is a saddle.  Because of the previous lemma, every orbit is bounded.  Moreover, every wandering orbit has its $\alpha$- and 
$\omega$-limit contained in two different saddle pieces (see Lemma \ref{cyc}). Let $x_0$ be a wandering point, and let $\alpha(x_0)\subset \Lambda_0$ and $\omega(x_0)\subset \Lambda_1$.  
By definition of saddle piece, there exists $x_{-1}, x_1$ wandering points, $\Lambda_{-1}, \Lambda_2$ saddle pieces, such that $\alpha(x_{-1})\subset \Lambda_{-1}$, $\omega(x_{-1})\subset \Lambda_0$,
$\alpha(x_{1})\subset \Lambda_{1}$, $\omega(x_{1})\subset \Lambda_2$.  Moreover, the $\Lambda_i$'s are pairwise different by Lemma \ref{cyc}.  Proceeding inductively, one obtains
a sequence of pairwise different saddle pieces $(\Lambda_n)_{n\in \Z}$ such that for all $n\in \Z$ there exists $x_n$ a wandering point with  $\alpha(x_n)\subset \Lambda_n$ and 
$\omega(x_n)\subset \Lambda_{n+1}$.  Moreover, we may assume that the set $X=\overline {\cup_{n\in \Z}\Lambda _n}$ is compact, because otherwise we are done by the previous lemma and the shadowing property.
Therefore, if we take $x_n\in \Lambda_n$ for all $n$, there is an accumulation point $z$ that must belong to a basic set $\Lambda$.  So, this basic set $\Lambda$ is accumulated
by basics sets.  Now, by Proposition 30 in \cite{dlrw} the basics sets are open in $\Omega(f)$  a contradiction.

\end{proof}

\section{Genus zero and finite type}

From now on, we will assume that there exists an attracting basic piece $\Lambda$.  The results are analogous assuming the existence of a repelling basic piece.  
In this section we make the assumption that $S$ has genus zero and finite type.  In this case,
we can compactify $S$ adding a point to each correspoding puncture, obtaining a compact surface $S'$, and $f$ extends trivially to a homeomorphism of $S'$ (which is topologically 
a sphere).
As $\Lambda\subset S'$ is an expansive 
transitive attractor,  by Theorem 1 in \cite{BM} then either
$\Lambda$ is  a single periodic orbit or it is derived from pseudo-Anosov.  Suppose that $\Lambda$ is  a single periodic orbit $x$.  If $p$ is the period of $x$, then
$f^p$ is $\ta$ and has an attracting fixed point. Therefore, by Lemma \ref{suf}, $S=\R^ 2$ and $f^p$ is conjugate to a homothety or reverse homothety.  If $x$ is not fixed,
then $f^p$ fixes the whole orbit of $x$, which is impossible because $f^p$ has only one fixed point.  Therefore, $x$ is fixed and $f$ is conjugate to a homothety or reverse homothety.

In what follows,  we will assume that $\Lambda$ is derived from pseudo-Anosov. We first note that we may assume $\Lambda\subset \R^2$, by taking a power 
of $f$ fixing every puncture of $S'$, and then taking $\R^2 = S'\backslash \{p\}$, $p$ a puncture.
It follows from \cite{tl} that $\Lambda$ separates the plane, as 1-dimensional non-separating plane continua are tree-like and do not support expansive homeomorphisms.

\begin{lemma}\label{punc}  There are no bounded connected components of $S\backslash \Lambda$.
 
\end{lemma}

\begin{proof}  Let $U$ be a bounded connected component of $S\backslash \Lambda$.  By Proposition 1.4 in \cite{BM}, there are only finitely many such components.  So, by taking
a power of $f$ if necessary, let us assume that $U$ is invariant.  Moreover, as $\Lambda$ is connected, $U$ is a topological disk with an attracting boundary.  As $\overline U$
is compact, there are only finitely many basic pieces inside $U$ (pieces are isolated by Propositions 27 and 30 in \cite{dlrw}).  As there are no cycles of saddles
(see Lemma \ref{cyc}), there must be a repelling basic piece  $\Lambda_1 \subset U$, which again must be separating.  Consider $U_1$ a connected component of $U\backslash \Lambda_1$
that is a topological disk with repelling boundary.  Then, there exists $\Lambda_2 \subset U_1$ an attracting piece.  Applying this inductively, as there
are finitely many pieces, one of them necessarilly does not separate the plane, and is therefore a periodic sink or source. 

Together with Lemma \ref{suf} this implies that every connected component of $S\backslash \Lambda$ is punctured. 
 
\end{proof}

The proof of the following lemma can be found in \cite{cgx} (proof of Lemma 10).

\begin{lemma}\label{autz}  Let $K$ be a compact invariant set with expansivity constant $C$.  Suppose that for all $x\in K$ there exists a neighborhood $U$ of $x$, and $z\in U$ such
that the orbit of $z$
$C/2$-shadows
any pseudo-orbit $(x_n)_{n\in \Z}$ such that $x_n = f^n (y), n<0$ for some $ y \in U$ and $x_n = f ^n(z), n\geq 0$.  Then, $K$ is finite.
 
\end{lemma}

\begin{lemma} The $\alpha$-limit of any point in the basin of attraction of $\Lambda$ is bounded.
 
\end{lemma}

\begin{proof}  Suppose there exists $y$ in the basin of attraction of $\Lambda$ such that $\alpha(y)$ is unbounded.  By Lemma \ref{epsilon} there exists $\epsilon: S \to \R$ a 
continuous positive map with the property that if $x\neq y$, then there 
exists 
$ n\in \N$ such that 
$d(f^{-n} (x),f^{-n} (y)) > \epsilon (f^{-n} (y))$.  Take $\delta: S \to \R$ as in the definition of shadowing, and for all $x\in \Lambda$ let $U=U_x= B (x,\epsilon(x)/2)$.  Note that
by transitivity of $\Lambda$ and the choice of $\epsilon$, there exists $n\geq 0$ such that $f^n(y)\in U$. Take $z=f^n (y)$ and note that it verifies the hypothesis of Lemma \ref{autz}.
 
\end{proof}

\begin{lemma}\label{ind} There exists a repelling basic piece in any connected component of $S\backslash \Lambda$ .
 
\end{lemma}

\begin{proof} The $\alpha$-limit of any point in the basin of attraction of $\Lambda$ is bounded, and therefore is contained in a basic piece.  If there is a chain of saddle pieces
going to infinity, by shadowing we obtain an orbit going to infinity, a contradiction.  The result follows.
 
\end{proof}

\noindent{\bf Proof of Theorem \ref{t1}}
We know by \cite{m2} that any attracting or repelling basic piece separates $S$ in at least three connected components.  Moreover, Lemma \ref{punc} implies that any such component is
punctured.  Let $U$ be a connected component of $S\backslash \Lambda$, let $x_0$ be a puncture in $U$, and let $\Lambda _1\subset U$ be the repelling basic piece given by Lemma 
\ref{ind}. Note that $\Lambda_1$ separates $U$ in at least three punctured connected components, so we can get a puncture $x_1\neq x_0$ in a connected component of 
$U\backslash \Lambda_1$. Proceeding inductively, we obtain an infinite sequence of different punctures, a contradiction.

\end{document}